\documentclass[12pt,reqno]{amsproc}

\usepackage{mathrsfs}

\usepackage{amsmath,amsthm,amssymb,wasysym,graphicx,mathtools}
\graphicspath{{./pics/}}
\usepackage[inline]{enumitem}

\usepackage{caption}
\usepackage{subcaption}

\numberwithin{equation}{section}
\theoremstyle{plain}

\newtheorem{theorem}{Theorem}[section]

\theoremstyle{definition}

\allowdisplaybreaks

\begin{document}\vspace*{-1cm}
\title{Starlikeness for Certain Close-to-Star Functions}

\author{R. Kanaga}
\address{Department of Mathematics \\National Institute of Technology\\Tiruchirappalli-620015,  India }
\email{kanagavennila75@gmail.com}

\author{V. Ravichandran}
\address{Department of Mathematics \\National Institute of Technology\\Tiruchirappalli-620015,  India }
\email{vravi68@gmail.com; ravic@nitt.edu}


\thanks{The first author is supported by  the institute  fellowship from NIT Tiruchirappalli.}
\maketitle
\begin{abstract}
We find the radius of starlikeness of order $\alpha$, $0\leq \alpha<1$,
of normalized analytic functions $f$ on the unit disk  satisfying
either $\operatorname{Re}(f(z)/g(z))>0$
or $\left| (f(z)/g(z))-1\right|<1$ for some  close-to-star  function $g$
with $\operatorname{Re}(g(z)/(z+z^2/2))>0$ as well as of the class
of  close-to-star functions $f$ satisfying
$\operatorname{Re}(f(z)/(z+z^2/2))>0$.  Several other radii   such as
radius of univalence and  parabolic starlikeness  are shown to be the same as the
radius of starlikeness of appropriate order.
\end{abstract}

\section{Introduction}
Let $\mathbb {D}_{r} :=\{ z\in{\mathbb{C}}:\left|z\right| <r\}$ and  $\mathbb{D}:=\mathbb{D}_1$. Let $\mathcal{A}:=\{f:\mathbb{D}\to \mathbb{C}| f \mbox{ is analytic}, f(0)=f'(0)-1=0\}$ and let $ \mathcal{S}$ be its subclass consisting of univalent functions. For each $f\in\mathcal{S}$, we associate the function $s_{f}:\mathbb{D}\to  \mathbb{C}$ defined by $s_{f}(z)=zf'(z)/f(z)$. For $0\leq\alpha <1$, the class $\mathcal{S^{*}}(\alpha)$ of  functions starlike of order $\alpha$ is the subclass of $\mathcal{S}$ consisting of functions $f$ satisfying the inequality  $\operatorname{Re}(s_{f}(z))> \alpha$. The class $ \mathcal{S^{*}}:=\mathcal{S^{*}}(0)$ is the usual class of starlike functions. The image $k(\mathbb{D})$ of the Koebe function $k(z)=z/(1-z)^2$ is not convex but the image $k(\mathbb{D}_{2-\sqrt{3}})$ is convex and $2-\sqrt{3}$ is the largest such radius. This radius is known as the radius of convexity of the Koebe function. More generally, given a class of functions $\mathcal{F} $ and  another class  $\mathcal{G}$ characterised by a property $P$, the largest number $ R_{\mathcal{G}}(\mathcal{F})$ with $0\leq R_{\mathcal{G}}(\mathcal{F})\leq 1$ such that every function in $\mathcal{F}$
has the property $P$, in each disk $\mathbb {D}_{r}$ for each $r$ with  $0<r< R_{\mathcal{G}}(\mathcal{F}) $ is called the $\mathcal{G}$ radius of $\mathcal{F}$.  Kaplan \cite{keplan} introduced the class of close-to-convex functions $f$ satisfying $\operatorname{Re}(f'(z)/g'(z))> 0$ for some convex function $g$. In \cite{mac gregor1,mac gregor2}, MacGregor found the radius of starlikeness for the class of functions $f$ satisfying either $ \operatorname{Re}(f(z)/g(z))>0$ or $ \left|(f'(z)/g'(z))-1\right|<1$ for some $g\in \mathcal{S}$; related radius problems were discussed in \cite{jain,bajpai,Bhargava,Causey,Chen,Goel,Kowalczyk1,Kowalczyk2,Lecko,ratti1,Asha}. Reade \cite{reade} defined a function $f\in\mathcal{A}$, with $f(z)\neq 0$ for $z\in\mathbb{D}\setminus\{0\}$,  to be close-to-star if there exists a starlike function $g$ (not necessarily normalized) satisfying $\operatorname{Re}(f(z)/g(z))>0$.  The function $f(z)=z+z^2/2$ maps  $\mathbb{D}$ onto the domain bounded by the cardioid  $u+1/2=\cos t (1+\cos t )$ and $v=\sin t(1+\sin t)$ and therefore starlike in $\mathbb{D}$. This function $f$ also satisfies the inequality $|f'(z)-1|<1$ (which also implies univalence of $f$).  Using this starlike function, we introduce  the following three classes:
\begin{align*}
\mathcal {F}_{1}:&=\{ f\in{\mathcal{A}}:\operatorname{Re}(f(z)/g(z))>0,\quad\operatorname{Re}(g(z)/(z+z^2/2))>0\quad  \text{for some $g$ in }\mathcal{A}\} , \\
\mathcal {F}_{2}:&=\{ f\in{\mathcal{A}}:\left| (f(z)/g(z))-1\right|<1, \quad\operatorname{Re}(g(z)/(z+z^2/2))>0\quad  \text{for some $g$ in }\mathcal{A}\},\\ \shortintertext{and}
\mathcal {F}_{3}:&=\{ f\in{\mathcal{A}}: \operatorname{Re}(f(z)/(z+z^2/2))>0\} .
\end{align*}
These  classes  $\mathcal{F}_{1}$,  $\mathcal{F}_{2}$ and  $\mathcal{F}_{3}$ are nonempty. Indeed, if the functions $f_i:\mathbb{D}\to\mathbb{C}$, $i=1,2,3$, are  defined by
\begin{equation}\label{extremal}
f_{1}(z)=\frac{(1+z)^{2}(z+z^2/2)}{(1-z)^{2}}, \quad f_{2}(z)=\frac{(1+z)^{2}(z+z^2/2)}{(1-z)}
\end{equation} and
\begin{equation}\label{extremal1}
f_{3}(z)=\frac{(1+z)(z+z^2/2)}{(1-z)},
\end{equation}
then it follows that  $f_{i}$ belongs to the class $\mathcal{F}_{i}$; the functions $f_1$
and $f_2$ satisfy the respective condition with $g =f_{3} $.  It is also clear that the class $\mathcal{F}_3$ is a subclass of close-to-star functions while the classes $\mathcal{F}_1$ and $\mathcal{F}_2$ are not. The functions in these classes are also not necessarily univalent. Indeed, the radius of univalence   $R_{\mathcal{S}}(\mathcal{F}_1)\approx 0.210756 $,  $R_{\mathcal{S}}(\mathcal{F}_2)\approx 0.248032 $, $R_{\mathcal{S}}(\mathcal{F}_3)\approx  0.347296$ are respectively the
smallest positive zero of the polynomials $P_{i}$ given by
\begin{equation}\label{star1}
 P_{1}(r)=1-5r+r^{2}+r^{3},\quad P_{2}(r)=2-8r-r^{2}+3r^{3},
\end{equation}
  and
\begin{equation}\label{star3}
  P_{3}(r)=1-3r+r^{3}.
\end{equation} These radius are in fact the radius of starlikeness of the respective classes (see Theorems~\ref{th1},~\ref{th2} ~and~ \ref{th3}). The sharpness of these radii follows as the
derivative of  $f_{1}$, $f_{2}$ and $f_{3}$,  given by
\begin{equation}
f'_{1}(z)=\frac{(1+z)(1+5z+z^2-z^{3})}{(1-z)^{3}}, \quad
f'_{2}(z)=\frac{(1+z)(2+8z-z^2 -3 z^3)}{2(1-z)^{2}},
\end{equation}
and
\begin{equation}
f'_{3}(z)=\frac{1+3z-z^{3} }{(1-z)^{2}},
\end{equation}
clearly vanishes at   $z=-R_{\mathcal{S}}(\mathcal{F}_{i})$ for $i=1,2,3$ respectively.

 Several subclasses of starlike functions are defined through subordination.
An analytic function $f$ is subordinate to the analytic function $g$, written $f\prec g$, if   there exists an analytic function $\omega:\mathbb{D}\to  \mathbb{D}$  with $\omega(0)=0$ and $f(z)=g(\omega(z))$ for all $z\in\mathbb{D}$. For univalent superordinate function $g$, we have $f\prec g$ if  $f(\mathbb{D})\subset g(\mathbb{D})$ and $f(0)=g(0)$.  Consider the functions $\varphi_i:\mathbb{D}\to \mathbb{C}$ defined by   $\varphi_{1}(z):=\sqrt{z+1}$,  $\varphi_{2}(z):=e^{z}$, $\varphi_{3}(z):=1+(4/3)z+(2/3)z^{2}$, $\varphi_{4}(z):= 1+ \sin z $, $\varphi_{5}(z) =: z + \sqrt{1 + z^{2}}$, $\varphi_{6}(z):=1+((zk+z^2)/(k^2-kz))$ where $ k=\sqrt{2}+1$ and $\varphi_{7}(z):=1+(2 \left(\log ((1+\sqrt{z})/(1-\sqrt{z}))\right)^{2}/\pi^{2}) $.
For $\varphi=\varphi_{i},(i=1,2,\dotsc,7)$ the class $\mathcal{{S}}^{*}(\varphi):=\{f\in\mathcal{A}:s_{f}\prec \varphi\}$ respectively becomes $\mathcal{{S}^{*}_{L}}$, $\mathcal{S}^{*}_{e}$, $\mathcal{S}_{c}^{*}$, $\mathcal{S}_{\text{sin}}^{*}$,  $\mathcal{S}_{\leftmoon}^{*}$, $\mathcal{S}_{R}^{*}$ and $\mathcal{S}^{*}_{p}$;  these  classes were studied in \cite{ cho,sushil,rajni,raina,ronn,kanika,sokol}.
For these $\varphi_i$, we study the $\mathcal{S}^*(\varphi)$  radius  of the  classes $\mathcal{F}_i$, $i=1,2,3$ introduced above. For example, for the class $\mathcal{F}_1$,
we have shown that the radius of starlikeness of order $\alpha$, $0\leq \alpha<1$,
is the smallest positive root in (0,1) of the  equation
\[ (\alpha-2)r^{3}-(2\alpha+2)r^{2}+(10-\alpha)r+2\alpha-2=0.\] In addition to finding
radius of lemniscate starlikeness, we have also shown that  $R_{\mathcal{S} }=R_{\mathcal{S} ^{*}} $,
 $R_{\mathcal{S}^{*}(1/2)}=R_{\mathcal{S}_{P}^{*}} $,
 $R_{\mathcal{S}^{*}(1/e)}=R_{\mathcal{S}_{e}^{*}} $,
 $R_{\mathcal{S}^{*}(1-\sin1)}=R_{\mathcal{S}_{\sin}^{*}} $,
 $R_{\mathcal{S}^{*}(\sqrt{2}-1)}=R_{\mathcal{S}_{\leftmoon}^{*}} $,
 $R_{\mathcal{S}^{*}(2(\sqrt{2}-1))}=R_{\mathcal{S}_{R}^{*}} $ and
 $R_{\mathcal{S}^{*}(1/3)}=R_{\mathcal{S}_{C}^{*}} $.
Similar results have been proved for the other two classes.

\section{Radius Problem for $\mathcal{F}_{1}$}
For  the function $f\in \mathcal{F}_{1}$, we
first determine the disk containing the image of the   disk $\mathbb{D}_r$ under
the mapping $zf'(z)/f(z)$.  This is done by associating the function $f$ with
 suitable functions with positive real part and then applying
the inequality (see \cite[Lemma 2]{shah})
\begin{equation}\label{e}
\left | \frac{z p'(z)}{p(z)} \right | \leq
 \frac{2(1-\alpha)|z|}{(1-|z|)(1+(1-2\alpha)|z|)} \quad  (z\in \mathbb{D})
\end{equation}
for the  function $p$ in the class $\mathscr{P}(\alpha)$ of all analytic function
$p:\mathbb{D}\rightarrow\mathbb{C}$ with $p(0)=1$ and
$\operatorname{Re}p(z)>\alpha$. We also need to know the image of the
disk $\mathbb{D}_r$ under the transform $w(z)=(z+1)/(z+2)$. This is a
linear fractional transformation and it  maps the disk $\mathbb{D}_r$ onto the disk
 \begin{equation}\label{c}
\left| w(z)-\frac{2- r^{2}}{4-r^{2}}\right|\leq \frac{ r}{4-r^{2}}.
\end{equation}

Since $f\in \mathcal{F}_1$, there is a function $g\in \mathcal{A}$ such that $\operatorname{Re}(f(z)/g(z))>0$ and $\operatorname{Re}(g(z)/(z+z^2/2))>0$ for all $z\in \mathbb{D}$. Thus, the functions  $p_{1},p_{2}:\mathbb{D}\rightarrow\mathbb{C}$  defined by  $p_{1}(z)=f(z)/g(z)$ and
$p_{2}(z)=g(z)/(z+z^2/2)$ are functions in $\mathcal{P}(0)$ and
\begin{equation}\label{a}
f(z)=p_{1}(z)  p_{2}(z) \left(z+ z^{2}/2\right)\quad (z\in\mathbb{D}).
\end{equation}
 From  \eqref{a}, it follows that
\begin{equation}\label{b}
{\frac{zf'(z)}{f(z)}=\frac{zp_{1}'(z)}{p_{1}(z)}+\frac{zp_{2}'(z)}{p_{2}(z)}+\frac{2(z+1)}{z+2}}.
\end{equation}
Using \eqref{e} (with $\alpha=0$) and \eqref{c} in \eqref{b}, we see that the image of the
disk $\mathbb{D}_r$ under the mapping $zf'(z)/f(z)$  is contained in the disk
\begin{equation}\label{f}
\left| \frac{zf'(z)}{f(z)}-\frac{4-2r^{2}}{4-r^{2}}\right|
\leq \frac{6r(3-r^{2})}{(1-r^{2})(4-r^{2})}.
\end{equation}
From \eqref{f}, it readily follows that
\begin{equation}\label{g}
\operatorname{Re} \left( \frac{zf'(z)}{f(z)}\right) \geq
\frac{4-2r^{2}}{4-r^{2}}-\frac{6r(3-r^{2})}{(1-r^{2})(4-r^{2})}
=\frac{2(1-5r+r^{2}+r^{3})}{(2-r)(1-r^{2})}  \quad  (\left| z \right| \leq r).
\end{equation}

Let  $R_{\mathcal{S}^{*}}\approx  0.2108$ be the unique zero in (0,1)
of the polynomial $1-5r+r^{2}+r^{3}$. Then,
 for every function $f\in \mathcal{F}_{1}$,  the inequality \eqref{g} shows that
  $\operatorname{Re}(s_{f}(z))>0$ in
each disk $\mathbb{D}_{r}$, for $0\leq r< R_{\mathcal{S}^{*}}$.
For the function $f_{1}$ defined in \eqref{extremal}, we have
\begin{equation}\label{th1b1}
s_{f_{1}}(z)= \frac{zf_{1}'(z)}{f_{1}(z)} =\frac{2(1+5z+z^2-z^3)}{(2+z)(1-z^{2})}
\end{equation}
and hence  $\operatorname{Re}(s_{f_{1}}(z))$ vanishes at $z=-R_{\mathcal{S}^{*}}$.
Thus, the radius of starlikeness  $R_{\mathcal{S}^{*}}$ for the class
$\mathcal{F}_{1}$ is the unique positive zero in $ (0,1)$ of the polynomial
$P_{1}$ defined in \eqref{star1} and is the same as the radius
of univalence  $R_{\mathcal{S}}$. Using the inequality \eqref{f}, we now determine     ${\mathcal{S^{*}}(\alpha)}$,
${\mathcal{{S}^{*}_{L}}}$, ${\mathcal{S}^{*}_{P}}$,
${\mathcal{S}^{*}_{e}}$, ${\mathcal{S}_{c}^{*}}$, ${\mathcal{S}_{\sin}^{*}}$,
${\mathcal{S}_{\leftmoon}^{*}}$ and ${\mathcal{S}_{R}^{*}} $
radii  for the class $\mathcal{F}_{1}$.
\begin{theorem}\label{th1}
The following sharp radii results  hold for the class $\mathcal{F}_{1}$:
\begin{enumerate}[label=(\roman*)]
\item For any  $0\leq \alpha<1$, the radius $R_{\mathcal{S^{*}}(\alpha)}$ is the smallest positive root of the  equation
\begin{equation}\label{1}
(\alpha-2)r^{3}-(2\alpha+2)r^{2}+(10-\alpha)r+2\alpha-2=0.
\end{equation}

\item The radius $R_{\mathcal{{S}^{*}_{L}}}$ ($\approx 0.0918$) is the smallest positive root  of the  equation
\begin{equation}
\label{2} (2-\sqrt{2})r^{3}-(2+2\sqrt{2})r^{2}+(\sqrt{2}-10)r+2\sqrt{2}-2=0.
\end{equation}

\item The radius $R_{\mathcal{{S}^{*}_{P}}}$ ($\approx 0.1092)$ is the same as
$R_{\mathcal{S}^{*}(1/2)}$.

\item The radius $R_{\mathcal{S}^{*}_{e}}$ ($ \approx 0.1370)$  is the same as  $R_{\mathcal{S}^{*}(1/e)}$.

\item The radius $R_{\mathcal{S}_{\sin}^{*}}$ ($\approx 0.17969)$  is the same as  $R_{\mathcal{S}^{*}(1-\sin1)}$.

\item The radius $R_{\mathcal{S}_{\leftmoon}^{*}}$ ( $\approx 0.12734)$ is  the same as $R_{\mathcal{S}^{*}(\sqrt{2}-1)}$.

\item The radius $R_{\mathcal{S}_{R}^{*}}$ ($\approx 0.0380$) is the same as  $R_{\mathcal{S}^{*}(2(\sqrt{2}-1))}$.

\item The radius $R_{\mathcal{S}_{C}^{*}}$ ($\approx 0.14418)$ is the same as  $R_{\mathcal{S}^{*}(1/3)}$.

\end{enumerate}\end{theorem}

\begin{proof}
\begin{enumerate}[label=(\roman*),  leftmargin=12pt]
\item Let the function $f\in$ $\mathcal{F}_{1}$ and $\alpha\in [0,1)$. Let $R:=R_{\mathcal{S^{*}}(\alpha)}$
be the  smallest positive root of the  equation \eqref{1} so that
\begin{equation}\label{th1p1}
 2(1-5R+R^{2}+R^{3})= \alpha(2-R)(1-R^{2}).
\end{equation}
The function
\begin{equation*}
h(r)= \frac{2(1-5r+r^{2}+r^{3})}{(2-r)(1-r^{2})}
\end{equation*}
is decreasing in $[0,1)$ and hence, for $ 0\leq r< R$,
we have, using \eqref{g} and \eqref{th1p1},
\begin{equation*}
\operatorname{Re} \left( \frac{zf'(z)}{f(z)}\right) \geq\frac{2(1-5r+r^{2}+r^{3})}{(2-r)(1-r^{2})}>
\frac{2(1-5R+R^{2}+R^{3})}{(2-R)(1-R^{2})}= \alpha   .
\end{equation*}
This proves that the function $f$ is starlike of order $\alpha $
in each disk $\mathbb{D}_{r}$ for $0\leq r< R $.
At the point $z=-R$, it can be seen, using \eqref{th1b1} and  \eqref{th1p1},  that
the function $f_{1}$ defined in \eqref{extremal} satisfies
\begin{equation*}
\operatorname{Re}\left( \frac{zf_{1}'(z)}{f_{1}(z)}\right)
=\frac{2(1-5R+R^2-R^3)}{(2-R)(1-R^{2})} =\alpha.
\end{equation*}
This shows that the radius $R$  is the sharp  radius of
starlikeness of order $\alpha $ of the class $\mathcal{F}_{1}$.

\item Let  $R:=R_{\mathcal{{S}^{*}_{L}}}$  be the smallest  positive root
of the equation \eqref{2} so that
\begin{equation}\label{th1p2}
2(1+5R+R^{2}-R^{3})=\sqrt{2}(1-R^2)(2+R).
\end{equation} Since the function
\[ h(r):=\frac{6r(3-r^{2})}{(1-r^{2})(4-r^{2})}+ \frac{4-2r^{2}}{4-r^{2}}
= \frac{2(1+5r+r^{2}-r^{3})}{(2+r)(1-r^2)}\]
is an increasing  function of $r$ in [0,1), it follows that, for  $0\leq r < R$,
$h(r)<h(R)=\sqrt{2}$ and hence, for  $0\leq r < R$, we have
\begin{align}\label{z}
  \frac{6r(3-r^{2})}{(1-r^{2})(4-r^{2})}&< \sqrt{2}- \frac{4-2r^{2}}{4-r^{2}}.
\end{align}
From \eqref{f} and \eqref{z}, we obtain
\begin{equation*}
\left|\frac{zf'(z)}{f(z)}-\frac{4-2r^{2}}{4-r^{2}} \right|< \sqrt{2}-\frac{4-2r^{2}}{4-r^{2}}  \quad  (\left| z \right| \leq r).
\end{equation*}
For  $0\leq r< R$, the center of the above disk
$c(r)= (4-2r^{2})/(4-r^{2})$ (being a decreasing function of $r$ on [0,1])
lies in the interval $[c(R),1]\subset (c(0.1), 1]
\approx (.997494,1]\subset [2\sqrt{2}/3,\sqrt{2})$. When $a\in [2\sqrt{2}/3,\sqrt{2})$, by \cite[Lemma 2.2]{lemniscate},
the disk $|w-a|<\sqrt{2}-a$ is contained in the lemniscate region $|w^2-1|<1$ and
hence,  for  $0\leq r< R$, we have
\begin{equation*}
\left|\left({\frac{zf'(z)}{f(z)}}\right)^{2}-1\right|< 1.
\end{equation*}
Thus, the radius of lemniscate starlikeness of the class $\mathcal{F}_1$
is at least $R$.  To show that the radius $R$  is sharp, using \eqref{th1b1} and \eqref{th1p2}, we see  that the function
$f_{1}$ defined in $\eqref{extremal}$ satisfies, at $z=R$,
\[ \frac{zf_{1}'(z)}{f_{1}(z)} =\frac{2(1+5R+R^2-R^3)}{(2+R)(1-R^{2})}=\sqrt{2} \]
and therefore
\begin{equation}
\left|\left({\frac{zf_{1}'(z)}{f_{1}(z)}}\right)^{2}-1\right|= 1.
\end{equation}
\item The number $R:=R_{\mathcal{{S}^{*}_{P}}}$ is the smallest  positive root  of the equation %\eqref{3} so that
 \begin{equation}\label{th1p3}
 2(1-5R+R^{2}+R^{3})= (1/2)(2-R)(1-R^{2}).
 \end{equation}
Since the function
\[ h(r):=\frac{6r(r^{2}-3)}{(1-r^{2})(4-r^{2})}+ \frac{4-2r^{2}}{4-r^{2}}
= \frac{2(1-5r+r^{2}+r^{3})}{(2-r)(1-r^2)}\]
is decreasing  function of $r$ in [0,1), it follows that,
$h(r)>h(R)=1/2$  for  $0\leq r < R$, and hence, for  $0\leq r < R$, we have
\begin{align}\label{z1}
\frac{6r(3-r^{2})}{(1-r^{2})(4-r^{2})} <\frac{4-2r^{2}}{4-r^{2}}-\frac{1}{2}.
\end{align}
From \eqref{f} and \eqref{z1}, we get
\begin{align}\label{z2}
\left|\frac{zf'(z)}{f(z)}-\frac{4-2r^{2}}{4-r^{2}} \right|< \frac{4-2r^{2}}{4-r^{2}}-\frac{1}{2}  \quad  (\left| z \right| \leq r).
\end{align}
For  $0\leq r< R$, the center of the above disk
$c(r)= (4-2r^{2})/(4-r^{2})$ (being a decreasing function of $r$ on [0,1])
lies in the interval $[c(R),1]\subset (c(0.2), 1]
\approx (.989899,1]\subset (1/2,3/2)$. When
$a\in (1/2,3/2)$, by \cite[Lemma 2.2]{shan},
the disk $|w-a|<a-(1/2)$ is contained in the parabolic region $|w-1|<\operatorname{Re}(w)$ and
hence,  for  $0\leq r< R$,  we have
 \begin{align}\label{9}
\left|{\frac{zf'(z)}{f(z)}}-1\right|<\operatorname{Re}\left(\frac{zf'(z)}{f(z)}\right) \quad  (\left| z \right| \leq r).
\end{align}
Thus, the radius of parabolic starlikeness of the class $\mathcal{F}_1$
is at least $R$.  To show that the radius $R$  is sharp, using \eqref{th1b1} and \eqref{th1p3}, we see  that the function
$f_{1}$ defined in $\eqref{extremal}$ satisfies, at $z=-R$,
\[ \frac{zf_{1}'(z)}{f_{1}(z)} =\frac{ 2(1-5R+R^{2}+R^{3})}{(2-R)(1-R^{2})}=\frac{1}{2} \]
and therefore
\begin{align}
\left|{\frac{zf_{1}'(z)}{f_{1}(z)}}-1\right|=\frac{1}{2}=\operatorname{Re}\left({\frac{zf_{1}'(z)}{f_{1}(z)}} \right).
\end{align}

\item The number $R:=R_{\mathcal{S}^{*}_{e}}$  is the smallest  positive root  of the equation
\begin{equation}\label{th1p4}
2(1-5R+R^{2}+R^{3})= (1/e)(2-R)(1-R^{2}).
\end{equation}
Since the function
\[ h(r):=\frac{6r(r^{2}-3)}{(1-r^{2})(4-r^{2})}+ \frac{4-2r^{2}}{4-r^{2}}
= \frac{2(1-5r+r^{2}+r^{3})}{(2-r)(1-r^2)}\]
is decreasing  function of $r$ in [0,1), it follows that
$h(r)>h(R)=1/e$  for  $0\leq r < R$ and hence, for  $0\leq r < R$,  we have
\begin{align}\label{z3}
\frac{6r(3-r^{2})}{(1-r^{2})(4-r^{2})}<\frac{4-2r^{2}}{4-r^{2}}-\frac{1}{e}.
\end{align}
From \eqref{f} and \eqref{z3}, we get
\begin{align*}
\left|\frac{zf'(z)}{f(z)}-\frac{4-2r^{2}}{4-r^{2}} \right|<\frac{4-2r^{2}}{4-r^{2}}-\frac{1}{e}  \quad  (\left| z \right| \leq r).
\end{align*}
For  $0\leq r< R$, the center of the above disk
$c(r)= (4-2r^{2})/(4-r^{2})$ (being a decreasing function of $r$ on [0,1])
lies in the interval $[c(R),1]\subset (c(0.2), 1]
\approx (.989899,1]\subset (e^{-1},(e+e^{-1})/2$]. When
$a\in (e^{-1},(e+e^{-1})/2]$, by \cite[Lemma 2.2]{rajni},
the disk $|w-a|<a-e^{-1}$ is contained in the  region $|\log w|<1$ and
hence,  for  $0\leq r< R$,
\begin{align}\label{10}
\left|\operatorname{log}\left(\frac{zf'(z)}{f(z)}\right) \right|< 1  \quad  (\left| z \right| \leq r).
\end{align}
Thus, the radius of exponential starlikeness of the class $\mathcal{F}_1$
is at least $R$.  To show that the radius $R$ is sharp, using \eqref{th1b1} and \eqref{th1p4}, we see  that the function
$f_{1}$ defined in $\eqref{extremal}$ satisfies, at $z=-R$,
\[ \frac{zf_{1}'(z)}{f_{1}(z)} =\frac{2(1-5R+R^{2}+R^{3})}{(2+R)(1-R^{2})}=\frac{1}{e} \]
and therefore
 \begin{align}
\left|\operatorname{log}\left(\frac{zf_{1}'(z)}{f_{1}(z)}\right) \right|=1 .
\end{align}

\item The number $R:=R_{\mathcal{S}_{\sin}^{*}}$ is the smallest  positive root  of the equation
\begin{equation}\label{th1p5}
2(1-5R+R^{2}+R^{3})= (1-\sin1)(2-R)(1-R^{2}).
\end{equation}
Since the function
\[ h(r):=\frac{6r(r^{2}-3)}{(1-r^{2})(4-r^{2})}+ \frac{4-2r^{2}}{4-r^{2}}
= \frac{2(1-5r+r^{2}+r^{3})}{(2-r)(1-r^2)}>1-\sin1\]
is decreasing  function of $r$ in [0,1), it follows that,
$h(r)>h(R)=1-\sin1$ for  $0\leq r < R$ and hence, for  $0\leq r < R$,  we have
\begin{align}\label{z4}
\frac{6r(3-r^{2})}{(1-r^{2})(4-r^{2})}<\frac{4-2r^{2}}{4-r^{2}}+\text{sin1}-1.
\end{align}
From $\eqref{f}$ and $\eqref{z4}$, we get
\begin{align*}
\left|\frac{zf'(z)}{f(z)}-\frac{4-2r^{2}}{4-r^{2}} \right|<\text{sin1} -1+\frac{4-2r^{2}}{4-r^{2}}\quad  (\left| z \right| \leq r).
\end{align*}
For  $0\leq r< R$, the center of the above disk
$c(r)= (4-2r^{2})/(4-r^{2})$ (being a decreasing function of $r$ on [0,1])
lies in the interval $[c(R),1]\subset (c(0.2), 1]
\approx (.989899,1]\subset (-1-\sin1, 1-\sin1)$. When
$a\in (-1-\sin1, 1+\sin1)$, by \cite[Lemma 3.3]{cho},
the disk $|w-a|<\sin1-|a-1|$ is contained in the region $\varphi_{4}(\mathbb{D})$, where $\varphi_{4}(z)=1+\sin z$ and
hence,  for  $0\leq r< R$,
${s_{f}}(\mathbb{D}_{r})\subset\varphi_{4}(\mathbb{D})$.
Thus, the radius of sine starlikeness of the class $\mathcal{F}_1$
is at least $R$.  To show that the radius $R$  is sharp, using \eqref{th1b1} and \eqref{th1p5}, we see  that the function $f_{1}$ defined in \eqref{extremal} satisfies, at $z=-R$,
\[ \frac{zf_{1}'(z)}{f_{1}(z)} =\frac{2(1-5R+R^{2}+R^{3})}{(2-R)(1-R^{2})}
=1-\sin1=\varphi_{4}(-1) \in \partial \varphi_{4}(\mathbb{D}). \]

\item The number $R:=R_{\mathcal{S}_{\leftmoon}^{*}}$ is the smallest  positive root  of the equation
\begin{equation}\label{th1p6}
2(1-5R+R^{2}+R^{3})= (\sqrt{2}-1)(2-R)(1-R^{2}).
\end{equation} Since the function
\[ h(r):=\frac{6r(r^{2}-3)}{(1-r^{2})(4-r^{2})}+ \frac{4-2r^{2}}{4-r^{2}}
= \frac{2(1-5r+r^{2}+r^{3})}{(2-r)(1-r^2)}\]
is decreasing  function of $r$ in [0,1), it follows that
$h(r)>h(R)=\sqrt{2}-1$  for  $0\leq r < R$ and hence, for  $0\leq r < R$, we have
\begin{align}\label{z5}
\frac{6r(3-r^{2})}{(1-r^{2})(4-r^{2})}<\frac{4-2r^{2}}{4-r^{2}}+1-\sqrt{2}.
\end{align}
From \eqref{f} and \eqref{z5}, we get
\begin{align*}
\left|\frac{zf'(z)}{f(z)}-\frac{4-2r^{2}}{4-r^{2}} \right|<\frac{4-2r^{2}}{4-r^{2}}+1-\sqrt{2}  \quad  (\left| z \right| \leq r).
\end{align*}
For  $0\leq r< R$, the center of the above disk
$c(r)= (4-2r^{2})/(4-r^{2})$ (being a decreasing function of $r$ on [0,1])
lies in the interval $[c(R),1]\subset (c(0.2), 1]
\approx (.9898991,1]\subset (\sqrt{2}$-1,$\sqrt{2}+1$). When
$a\in(\sqrt{2}-1,\sqrt{2}+1)$, by \cite[Lemma 2.1]{gandhi},
the disk $|w-a|<1-|\sqrt{2}-a|$ is contained in the  region $|w^{2}-1|<2|w|$ and
hence,  for  $0\leq r< R$,
\begin{align}\label{11}
\left|\left({\frac{zf'(z)}{f(z)}}\right)^{2}-1\right|< 2\left|\left({\frac{zf'(z)}{f(z)}}\right)\right|  \quad  (\left| z \right| \leq r).
\end{align}
Thus, the radius of lune starlikeness of the class $\mathcal{F}_1$
is at least $R$.  To show that the radius $R$  is sharp, using \eqref{th1b1} and \eqref{th1p6}, we see  that the function
$f_{1}$ defined in $\eqref{extremal}$ satisfies
\[ \frac{zf_{1}'(z)}{f_{1}(z)} =\frac{2(1-5R+R^{2}+R^{3})}{(2-R)(1-R^{2})}=\sqrt{2}-1 \]
at $z=-R$ and therefore
\begin{align}
 \left|\left({\frac{zf_{1}'(z)}{f_{1}(z)}}\right)^{2}-1\right|= 2\left|\left({\frac{zf_{1}'(z)}{f_{1}(z)}}\right)\right|.
\end{align}

\item The number $R:=R_{\mathcal{{S}^{*}_{R}}}$  is the smallest  positive root  of the equation
\begin{equation}\label{th1p7}
 2(1-5R+R^{2}+R^{3})= 2(\sqrt{2}-1)(2-R)(1-R^{2}).
\end{equation}
 Since the function
 \[ h(r):=\frac{6r(r^{2}-3)}{(1-r^{2})(4-r^{2})}+ \frac{4-2r^{2}}{4-r^{2}}
 = \frac{2(1-5r+r^{2}+r^{3})}{(2-r)(1-r^2)}\]
 is decreasing  function of $r$ in [0,1), it follows that
 $h(r)>h(R)=2(\sqrt{2}-1)$ for  $0\leq r < R$ and hence, for  $0\leq r < R$, we have
\begin{align}\label{z6}
  \frac{6r(3-r^{2})}{(1-r^{2})(4-r^{2})}< \frac{4-2r^{2}}{4-r^{2}}-2(\sqrt{2}-1).
\end{align}
From \eqref{f} and \eqref{z6}, we get
\begin{align*}
\left|\frac{zf'(z)}{f(z)}-\frac{4-2r^{2}}{4-r^{2}} \right|< \frac{4-2r^{2}}{4-r^{2}}-2(\sqrt{2}-1)  \quad  (\left| z \right| \leq r).
\end{align*}
For  $0\leq r< R$, the center of the above disk
$c(r)= (4-2r^{2})/(4-r^{2})$ (being a decreasing function of $r$ on [0,1])
lies in the interval $[c(R),1]\subset (c(0.1), 1]
\approx (.99741,1]\subset (2(\sqrt{2}-1), \sqrt{2}]$. When
$a\in  (2(\sqrt{2}-1), \sqrt{2}]$, by \cite[Lemma 2.2]{sushil},
the disk $|w-a|<a-2(\sqrt{2}-1)$ is contained in the region $\varphi_{6}(\mathbb{D})$, where $\varphi_{6}(z):=1+(zk+z^2/(k^2-kz))$ and $k=\sqrt{2}+1$.
Hence,  for  $0\leq r< R$,
${s_{f}}(\mathbb{D}_{r})\subset\varphi_{6}(\mathbb{D})$.
Thus, the radius of the class $\mathcal{F}_1$
is at least $R$.  To show that the radius $R$  is sharp, using \eqref{th1b1} and \eqref{th1p7}, we see  that the function $f_{1}$ defined in $\eqref{extremal}$ satisfies, at $z=-R$,
\[ \frac{zf_{1}'(z)}{f_{1}(z)} =\frac{2(1-5R+R^{2}+R^{3})}{(2+R)(1-R^{2})}=2(\sqrt{2}-1)= \varphi_{6}(-1)\in \partial \varphi_{6}(\mathbb{D}). \]

\item The number $R:=R_{\mathcal{{S}^{*}_{C}}}$  is the smallest  positive root  of the equation
\begin{equation}\label{th1p8}
2(1-5R+R^{2}+R^{3})= (1/3)(2-R)(1-R^{2}).
\end{equation}
Since the function
\[ h(r):=\frac{6r(r^{2}-3)}{(1-r^{2})(4-r^{2})}+ \frac{4-2r^{2}}{4-r^{2}}
 = \frac{2(1-5r+r^{2}+r^{3})}{(2-r)(1-r^2)}\]
  is decreasing  function of $r$ in [0,1), it follows that
  $h(r)>h(R)=1/3$ for  $0\leq r < R$ and hence, for  $0\leq r < R$, we have
\begin{align}\label{z7}
  \frac{6r(3-r^{2})}{(1-r^{2})(4-r^{2})}-\frac{4-2r^{2}}{4-r^{2}}<-\frac{1}{3}.
\end{align}
From \eqref{f} and \eqref{z7}, we get
\begin{align*}
\left|\frac{zf'(z)}{f(z)}-\frac{4-2r^{2}}{4-r^{2}} \right|<\frac{4-2r^{2}}{4-r^{2}}-\frac{1}{3}  \quad  (\left| z \right| \leq r) .
\end{align*}
For  $0\leq r< R$, the center of the above disk
$c(r)= (4-2r^{2})/(4-r^{2})$ (being a decreasing function of $r$ on [0,1])
lies in the interval $[c(R),1]\subset (c(0.2), 1] \approx (.989899,1]\subset (1/3,5/3)$. When
$a\in  (1/3,5/3)$, by \cite[Lemma 2.5]{kanika}, the disk $|w-a|<a-1/3$ is lies in the cardioid region $\varphi_{3}(\mathbb{D})$. Hence,  for  $0\leq r< R$, ${s_{f}}(\mathbb{D}_{r})\subset\varphi_{3}(\mathbb{D})$.
Thus, the radius of cardioid starlikeness of the class $\mathcal{F}_1$
is at least $R$.  To show that the radius $R$  is sharp, using \eqref{th1b1} and \eqref{th1p8}, we see  that the function $f_{1}$ defined in \eqref{extremal} satisfies, at $z=-R$, \[ \frac{zf_{1}'(z)}{f_{1}(z)}
=\frac{2(1-5R+R^{2}+R^{3})}{(2-R)(1-R^{2})}=1/3=\varphi_{3}(-1)\in \partial \varphi_{3}(\mathbb{D}). \qedhere\]\end{enumerate}
\end{proof}

\section{Radius Problem for $\mathcal{F}_{2}$}
For the function $f\in \mathcal{F}_{2}$, there is a function $g\in \mathcal{A}$ such that $\operatorname{Re}({{g(z)}/{f(z)}})>1/2$ and $\operatorname{Re}({2g(z)}/{(z+z^{2}/2}))>0 $.
The functions $p_{1},p_{2}:\mathbb{D}\to  \mathbb{C}$ defined by $p_{1}(z)=g(z)/f(z)$,  $p_{2}(z)=g(z)/(z+z^2/2)$ are the functions in $\mathscr{P}(1/2)$ and $\mathscr{P}(0)$ respectively and
\begin{equation}\label{a1}
f(z)=(p_{2}(z)/p_{1}(z))(z+z^{2}/2) \quad( z\in \mathbb{D}).
\end{equation}
From $\eqref{a1}$, it follows that
\begin{equation}\label{b1}
{\frac{zf'(z)}{f(z)}=\frac{zp_{2}'(z)}{p_{2}(z)}-\frac{zp_{1}'(z)}{p_{1}(z)}+\frac{2(z+1)}{z+2}}.
\end{equation}
Using \eqref{e} and \eqref{c} in \eqref{b1}, we see that the image of the disk $\mathbb{D}_r$ under the mapping $zf'(z)/f(z)$  is contained in the disk
\begin{equation}\label{f1}
\left| \frac{zf'(z)}{f(z)}-\frac{4-2r^{2}}{4-r^{2}}\right|\leq \frac{r(14+4r-5r^{2}-r^{3})}{(1-r^{2})(4-r^{2})}.
\end{equation}
From \eqref{f1}, it readily follows that
\begin{equation}\label{g1}
\operatorname{Re} \left( \frac{zf'(z)}{f(z)}\right)\geq \frac{4-2r^{2}}{4-r^{2}}- \frac{r(14+4r-5r^{2}-r^{3})}{(1-r^{2})(4-r^{2})}
= \frac{2-8r-r^{2}+3r^{3}}{(2-r)(1-r^{2})}  \quad  (\left| z \right| \leq r).
\end{equation}
Let  $R_{\mathcal{S}^{*}}\approx  0.248$ be the zero in (0,1)
of the polynomial $3r^{3}-r^{2}-8r+2$. Then,
for every function $f\in \mathcal{F}_{2}$,  the inequality \eqref{g1} shows that
$\operatorname{Re}(s_{f}(z))>0$ in
each disk $\mathbb{D}_{r}$, for $0\leq r< R_{\mathcal{S}^{*}}$.
For the function $f_{2}$ defined in \eqref{extremal}, we have
\begin{equation}\label{th2b1}
s_{f_{2}}(z)= \frac{zf_{2}'(z)}{f_{2}(z)} =\frac{2+8z-z^2-3z^3}{(2+z)(1-z^{2})}
\end{equation}
and hence  $\operatorname{Re}(s_{f_{2}}(z))$ vanishes at $z=-R_{\mathcal{S}^{*}}$.
Thus, the radius of starlikeness  $R_{\mathcal{S}^{*}}$ for the class
$\mathcal{F}_{2}$ is the smallest positive zero in $(0,1)$ of the polynomial
$P_{2}$ defined in \eqref{star1} and is the same as the radius
of univalence  $R_{\mathcal{S}}$. Using the inequality \eqref{f1}, we now determine     ${\mathcal{S^{*}}(\alpha)}$, ${\mathcal{S}^{*}_{P}}$,
${\mathcal{S}^{*}_{e}}$, ${\mathcal{S}_{c}^{*}}$, ${\mathcal{S}_{sin}^{*}}$,
${\mathcal{S}_{\leftmoon}^{*}}$ and ${\mathcal{S}_{R}^{*}} $
radii  for the class $\mathcal{F}_{2}$.

\begin{theorem}$\label{th2}$
The following sharp radii results  hold for the class   $\mathcal{F}_{2}$:
\begin{enumerate}[label=(\roman*)]
\item For any  $0\leq \alpha<1$, the radius $R_{\mathcal{S^{*}}(\alpha)}$ is the smallest positive root of the  polynomial
\begin{equation}\label{1a}
(3-\alpha)r^{3}+(2\alpha-1)r^{2}-(8-\alpha)r+2-2\alpha=0.
\end{equation}
\item The radius $R_{\mathcal{S}^{*}_{p}}$ ($\approx 0.1341)$ is the same as  $R_{\mathcal{S}^{*}(1/2)}$.

\item The radius $R_{\mathcal{S}^{*}_{e}} (\approx 0.16628)$ is the same as  $R_{\mathcal{S}^{*}(1/e)}$.

\item The radius $R_{\mathcal{S}_{\sin}^{*}}$ ($\approx 0.2142)$ is the same as $R_{\mathcal{S}^{*}(1-\sin1)}$.

\item The radius $R_{\mathcal{S}_{\leftmoon}^{*}}$ ($\approx 0.1551$) is the same as  $R_{\mathcal{S}^{*}(\sqrt{2}-1)}$.

\item The radius $R_{\mathcal{S}_{R}^{*}}$  ($\approx 0.0481$) is the same as  $R_{\mathcal{S}^{*}(2(\sqrt{2}-1))}$.

\item The radius $R_{\mathcal{S}_{c}^{*}}$ ($\approx 0.1744)$ is the same as  $R_{\mathcal{S}^{*}(1/3)}$.
\end{enumerate}
\end{theorem}
\begin{proof}
\begin{enumerate}[label=(\roman*),  leftmargin=12pt]
	
\item Let the function $f\in$ $\mathcal{F}_{2}$ and $\alpha\in [0,1)$. Let $R:=R_{\mathcal{S^{*}}(\alpha)}$
be the  smallest positive root of the  equation \eqref{1a} so that
\begin{equation}\label{th2p1}
2-8R-R^{2}+3R^{3}= \alpha(2-R)(1-R^{2}).
\end{equation}
The function
\begin{equation*}	
h(r)= \frac{2-8r-r^{2}+3r^{3}}{(2-r)(1-r^{2})}
\end{equation*}
is decreasing in $[0,1)$ and hence, for $ 0\leq r< R$, we have, using \eqref{g1} and \eqref{th2p1},
\begin{equation*}
\operatorname{Re}\left(\frac{zf'(z)}{f(z)}\right)\geq\frac{2-8r-r^{2}+3r^{3}}{(2-r)(1-r^{2})}> \frac{	2-8R-R^{2}+3R^{3}}{(2-R)(1-R^{2})}=\alpha \quad (0\leq r<R).
\end{equation*}
This proves that the function $f$ is starlike of order $\alpha $ in each disk $\mathbb{D}_{r}$ for $0\leq r< R $.
At the point $z=-R$, it can be seen, using \eqref{th2b1} and \eqref{th2p1}, that the function $f_{2}$ defined in \eqref{extremal} satisfies
\begin{equation*}
\operatorname{Re}\left( \frac{zf_{2}'(z)}{f_{2}(z)}\right)
=\frac{	2-8R-R^{2}+3R^{3}}{(2-R)(1-R^{2})}=\alpha.
\end{equation*}
This shows that the radius $R$ is the radius of starlikeness of order $\alpha $ of the class $\mathcal{F}_{2}$.
	
\item The number $R:=R_{\mathcal{{S}^{*}_{P}}}$  is the smallest  positive root  of the equation \begin{equation}\label{th2p2}
2-8R-R^{2}+3R^{3}= (1/2)(2-R)(1-R^{2}).
\end{equation}
Since the function
\[ h(r):=\frac{r(r^{3}+5r^{2}-4r-14)}{(1-r^{2})(4-r^{2})}+ \frac{4-2r^{2}}{4-r^{2}}
= \frac{2-8r-r^{2}+3r^{3}}{(2-r)(1-r^2)}\]
is decreasing  function of $r$ in [0,1), it follows that $h(r)>h(R)=1/2$ for  $0\leq r < R$ and hence, for  $0\leq r < R$, we have
\begin{align}\label{z8}
\frac{r(r^{3}+5r^{2}-4r-14)}{(1-r^{2})(4-r^{2})}+\frac{4-2r^{2}}{4-r^{2}}> \frac{1}{2}.
\end{align}
From \eqref{f1} and \eqref{z8}, we get
\begin{align*}
\left|\frac{zf'(z)}{f(z)}-\frac{4-2r^{2}}{4-r^{2}} \right|\leq\frac{4-2r^{2}}{4-r^{2}}-\frac{1}{2} \quad  (\left| z \right| \leq r).
\end{align*}
For  $0\leq r< R$, the center of the above disk $c(r)= (4-2r^{2})/(4-r^{2})$ (being a decreasing function of $r$ on [0,1]) lies in the interval $[c(R),1]\subset (c(0.2), 1] \approx (.9898991,1]\subset (1/2,3/2)$. When
$a\in (1/2,3/2)$, by \cite[Lemma 2.2]{shan}, the disk $|w-a|<a-(1/2)$ is contained in the parabolic region $|w-1|<\operatorname{Re}(w)$ and hence,  for  $0\leq r< R$,  we have
\begin{align}\label{9a}
\left|{\frac{zf'(z)}{f(z)}}-1\right|<\operatorname{Re}\left(\frac{zf'(z)}{f(z)}\right) \quad  (\left| z \right| \leq r) .
\end{align}
Thus, the radius of parabolic starlikeness of the class $\mathcal{F}_2$ is at least $R$.  To show that the radius $R$  is sharp, using \eqref{th2b1} and \eqref{th2p2}, we see  that the function $f_{2}$ defined in \eqref{extremal} satisfies
\[ \frac{zf_{2}'(z)}{f_{2}(z)} =\frac{	2-8R-R^{2}+3R^{3}}{(2-R)(1-R^{2})}=\frac{1}{2} \]
at $z=-R$ and therefore
\begin{align}
\left|{\frac{zf_{2}'(z)}{f_{2}(z)}}-1\right|=\operatorname{Re}\left({\frac{zf_{2}'(z)}{f_{2}(z)}} \right).
\end{align}

\item The number $R:=R_{\mathcal{S}^{*}_{e}}$  is the smallest  positive root  of the equation
\begin{equation}\label{th2p3}
2-8R-R^{2}+3R^{3}= (1/e)(2-R)(1-R^{2}).
\end{equation}
Since the function
\[ h(r):=\frac{r(r^{3}+5r^{2}-4r-14)}{(1-r^{2})(4-r^{2})}+ \frac{4-2r^{2}}{4-r^{2}} = \frac{2-8r-r^{2}+3r^{3}}{(2-r)(1-r^2)}\]
is decreasing  function of $r$ in [0,1), it follows that $h(r)>h(R)=1/e$ for  $0\leq r < R$ and hence, for  $0\leq r < R$, we have
\begin{align}\label{z9}
\frac{r(r^{3}+5r^{2}-4r-14)}{(1-r^{2})(4-r^{2})}+\frac{4-2r^{2}}{4-r^{2}}> \frac{1}{e}.
\end{align}
From \eqref{f1} and \eqref{z9}, we get
\begin{align*}
\left|\frac{zf'(z)}{f(z)}-\frac{4-2r^{2}}{4-r^{2}} \right|<\frac{4-2r^{2}}{4-r^{2}}-\frac{1}{e}  \quad  (\left| z \right| \leq r).
\end{align*}
For  $0\leq r< R$, the center of the above disk
$c(r)= (4-2r^{2})/(4-r^{2})$ (being a decreasing function of $r$ on [0,1])
lies in the interval $[c(R),1]\subset (c(0.2), 1]
\approx (.9898991,1]\subset (e^{-1},(e+e^{-1})/2$]. When
$a\in (e^{-1},(e+e^{-1})/2]$, by \cite[Lemma 2.2]{rajni},
the disk $|w-a|<a-e^{-1}$ is contained in the  region $|\log w|<1$ and hence,  for  $0\leq r< R$,  we have
\begin{align}\label{10a}
\left|\operatorname{log}\left(\frac{zf'(z)}{f(z)}\right) \right|< 1  \quad  (\left| z \right| \leq r).
\end{align}
Thus, the radius of exponential starlikeness of the class $\mathcal{F}_2$
is at least $R$.  To show that the radius $R$  is sharp, using \eqref{th2b1} and  \eqref{th2p3}, we see  that the function $f_{2}$ defined in \eqref{extremal} satisfies
\[ \frac{zf_{2}'(z)}{f_{2}(z)} =\frac{	2-8R-R^{2}+3R^{3}}{(2-R)(1-R^{2})}=\frac{1}{e} \]
at $z=-R$ and therefore
\begin{align}
\left|\operatorname{log}\left(\frac{zf_{2}'(z)}{f_{2}(z)}\right) \right|=1 .
\end{align}

\item The number $R:=R_{\mathcal{S}^{*}_{\sin}}$  is the smallest  positive root  of the equation
\begin{equation}\label{th2p}
 2-8R-R^{2}+3R^{3}= (1-\sin1)(2-R)(1-R^{2}).
\end{equation}
Since the function \[ h(r):=\frac{r(r^{3}+5r^{2}-4r-14)}{(1-r^{2})(4-r^{2})}+ \frac{4-2r^{2}}{4-r^{2}}
= \frac{2-8r-r^{2}+3r^{3}}{(2-r)(1-r^2)}\]
is decreasing  function of $r$ in [0,1), it follows that $h(r)>h(R)=1-\sin1$ for  $0\leq r < R$ and hence, for  $0\leq r < R$, we have
\begin{align}\label{z10}
\frac{r(r^{3}+5r^{2}-4r-14)}{(1-r^{2})(4-r^{2})}+\frac{4-2r^{2}}{4-r^{2}}> 1-\sin1.
\end{align}
From \eqref{f1} and \eqref{z10}, we get
\begin{align*}
\left|\frac{zf'(z)}{f(z)}-\frac{4-2r^{2}}{4-r^{2}} \right|\leq\frac{4-2r^{2}}{4-r^{2}}+\sin1-1  \quad  (\left| z \right| \leq r).
\end{align*}
For  $0\leq r< R$, the center of the above disk $c(r)= (4-2r^{2})/(4-r^{2})$ (being a decreasing function of $r$ on [0,1]) lies in the interval $[c(R),1]\subset (c(0.3), 1]
\approx (.976982,1]\subset (-1-\sin1, 1-\sin1)$. When $a\in (-1-\sin1, 1+\sin1)$, by \cite[Lemma 3.3]{cho},
the disk $|w-a|<\sin1-|a-1|$ is contained in the region $\varphi_{4}(\mathbb{D})$, where $\varphi_{4}(z)=1+\sin z$ and hence,  for  $0\leq r< R$, ${s_{f}}(\mathbb{D}_{r})\subset\varphi_{4}(\mathbb{D})$.
Thus, the radius of sine starlikeness of the class $\mathcal{F}_2$
is at least $R$.  To show that the radius $R$  is sharp, using \eqref{th2b1} and \eqref{th2p}, we see  that the function
$f_{2}$ defined in $\eqref{extremal}$ satisfies, at $z=-R$,
\[ \frac{zf_{2}'(z)}{f_{2}(z)} =\frac{	2-8R-R^{2}+3R^{3}}{(2-R)(1-R^{2})}=1-\sin1=\varphi_{4}(-1)\in \partial \varphi_{4}(\mathbb{D}). \]

\item The number $R:=R_{\mathcal{S}^{*}_{\leftmoon}}$  is the smallest  positive root  of the equation
 \begin{equation}\label{th2p5}
 	2-8R-R^{2}+3R^{3}= (\sqrt{2}-1)(2-R)(1-R^{2}).
 \end{equation}
 Since the function
 \[ h(r):=\frac{r(r^{3}+5r^{2}-4r-14)}{(1-r^{2})(4-r^{2})}+ \frac{4-2r^{2}}{4-r^{2}}
 = \frac{2-8r-r^{2}+3r^{3}}{(2-r)(1-r^2)}\]
  is decreasing  function of $r$ in [0,1), it follows that
  $h(r)>h(R)=\sqrt{2}-1$ for  $0\leq r < R$ and hence, for  $0\leq r < R$, we have
 \begin{align}\label{z11}
 \frac{r(r^{3}+5r^{2}-4r-14)}{(1-r^{2})(4-r^{2})}+\frac{4-2r^{2}}{4-r^{2}}> \sqrt{2}-1.
 \end{align}
   From \eqref{f1} and \eqref{z11}, we get
  \begin{align*}
  \left|\frac{zf'(z)}{f(z)}-\frac{4-2r^{2}}{4-r^{2}} \right|< \frac{4-2r^{2}}{4-r^{2}}+1-\sqrt{2}  \quad  (\left| z \right| \leq r).
\end{align*}
For  $0\leq r< R$, the center of the above disk
$c(r)= (4-2r^{2})/(4-r^{2})$ (being a decreasing function of $r$ on [0,1])
lies in the interval $[c(R),1]\subset (c(0.2), 1]
\approx (.9898991,1]\subset (\sqrt{2}$-1,$\sqrt{2}+1$). When
$a\in(\sqrt{2}-1,\sqrt{2}+1)$, by \cite[Lemma 2.1]{gandhi},
the disk $|w-a|<1-|\sqrt{2}-a|$ is contained in the  region $|w^{2}-1|<2|w|$ and
hence,  for  $0\leq r< R$,  we have
\begin{align}\label{11a}
\left|\left({\frac{zf'(z)}{f(z)}}\right)^{2}-1\right|< 2\left|\left({\frac{zf'(z)}{f(z)}}\right)\right|  \quad  (\left| z \right| \leq r).
\end{align}
  Thus, the radius of lune starlikeness of the class $\mathcal{F}_2$
  is at least $R$.  To show that the radius $R$  is sharp, using \eqref{th2b1} and \eqref{th2p5}, we see  that the function
  $f_{2}$ defined in $\eqref{extremal}$ satisfies
  \[ \frac{zf_{2}'(z)}{f_{2}(z)} =\frac{	2-8R-R^{2}+3R^{3}}{(2-R)(1-R^{2})}=\sqrt{2}-1 \]
 at $z=-R$  and therefore
  \begin{align}
  \left|\left({\frac{zf_{2}'(z)}{f_{2}(z)}}\right)^{2}-1\right|= 2\left|\left({\frac{zf_{2}'(z)}{f_{2}(z)}}\right)\right|.
  \end{align}

\item The number $R:=R_{\mathcal{S}^{*}_{R}}$  is the smallest  positive root  of the equation
\begin{equation}\label{th2p6}
 2-8R-R^{2}+3R^{3}= 2(\sqrt{2}-1)(2-R)(1-R^{2}).
\end{equation}
Since the function
\[ h(r):=\frac{r(r^{3}+5r^{2}-4r-14)}{(1-r^{2})(4-r^{2})}+ \frac{4-2r^{2}}{4-r^{2}} = \frac{2-8r-r^{2}+3r^{3}}{(2-r)(1-r^2)}\]
is decreasing  function of $r$ in [0,1), it follows that
$h(r)>h(R)=2(\sqrt{2}-1)$ for  $0\leq r < R$ and hence, for  $0\leq r < R$, we have
\begin{align}\label{z12}
  \frac{r(r^{3}+5r^{2}-4r-14)}{(1-r^{2})(4-r^{2})}+\frac{4-2r^{2}}{4-r^{2}}>2(\sqrt{2}-1)\quad(0\leq r< R).
\end{align}
From \eqref{f1} and \eqref{z12}, we get
\begin{align*}
\left|\frac{zf'(z)}{f(z)}-\frac{4-2r^{2}}{4-r^{2}} \right|<\frac{4-2r^{2}}{4-r^{2}}-2(\sqrt{2}-1)  \quad  (\left| z \right| \leq r).
\end{align*}
For  $0\leq r< R$, the center of the above disk
$c(r)= (4-2r^{2})/(4-r^{2})$ (being a decreasing function of $r$ on [0,1]) lies in the interval $[c(R),1]\subset (c(0.1), 1]
\approx (.99741,1]\subset (2(\sqrt{2}-1), \sqrt{2}]$. When $a\in  (2(\sqrt{2}-1), \sqrt{2}]$, by \cite[Lemma 2.2]{sushil},
the disk $|w-a|<a-2(\sqrt{2}-1)$ is contained in the region $\varphi_{6}(\mathbb{D})$, where $\varphi_{6}(z):=1+((zk+z^2)/(k^2-kz))$ and $k=\sqrt{2}+1$. Hence,  for  $0\leq r< R$,
${s_{f}}(\mathbb{D}_{r})\subset\varphi_{6}(\mathbb{D})$. Thus, the $\mathcal{S}^{*}_{R}$ radius of the class $\mathcal{F}_2$
is at least $R$.  To show that the radius $R$  is sharp, using \eqref{th2b1} and \eqref{th2p6}, we see  that the function $f_{2}$ defined in \eqref{extremal} satisfies, at $z=-R$,
\[ \frac{zf_{2}'(z)}{f_{2}(z)} =\frac{2-8R-R^{2}+3R^{3}}{(2-R)(1-R^{2})}=2(\sqrt{2}-1)=\varphi_{6}(-1)\in \partial \varphi_{6}(\mathbb{D}). \]

\item The number $R:=R_{\mathcal{S}^{*}_{C}}$  is the smallest  positive root  of the equation
\begin{equation}\label{th2p4}
2-8R-R^{2}+3R^{3}= (1/3)(2-R)(1-R^{2}).
\end{equation}
Since the function
\[ h(r):=\frac{r(r^{3}+5r^{2}-4r-14)}{(1-r^{2})(4-r^{2})}+ \frac{4-2r^{2}}{4-r^{2}}
= \frac{2-8r-r^{2}+3r^{3}}{(2-r)(1-r^2)}\]
is decreasing  function of $r$ in [0,1), it follows that
$h(r)>h(R)=1/3$ for  $0\leq r < R$ and hence, for  $0\leq r < R$, we have
\begin{align}\label{z13}
\frac{r(r^{3}+5r^{2}-4r-14)}{(1-r^{2})(4-r^{2})}+\frac{4-2r^{2}}{4-r^{2}}> \frac{1}{3}.
\end{align}
From \eqref{f1} and \eqref{z13}, we get
\begin{align*}
\left|\frac{zf'(z)}{f(z)}-\frac{4-2r^{2}}{4-r^{2}} \right|<\frac{4-2r^{2}}{4-r^{2}}-\frac{1}{3}  \quad  (\left| z \right| \leq r).
\end{align*}
For  $0\leq r< R$, the center of the above disk $c(r)= (4-2r^{2})/(4-r^{2})$ (being a decreasing function of $r$ on [0,1])
lies in the interval $[c(R),1]\subset (c(0.2), 1] \approx (.989899,1]\subset (1/3,5/3)$. When $a\in  (1/3,5/3)$, by \cite[Lemma 2.5]{kanika}, the disk $|w-a|<a-1/3$ is lies in the cardioid region $\varphi_{3}(\mathbb{D})$. Hence, for  $0\leq r< R$, ${s_{f}}(\mathbb{D}_{r})\subset\varphi_{3}(\mathbb{D})$. Thus, the $\mathcal{S}^{*}_{C}$ radius of the class $\mathcal{F}_2$ is at least $R$.  To show that the radius $R$ is sharp, using \eqref{th2b1} and  \eqref{th2p4}, we see  that the function $f_{2}$ defined in \eqref{extremal} satisfies, at $z=-R$, \[ \frac{zf_{2}'(z)}{f_{2}(z)}
=\frac{2-8R-R^{2}+3R^{3}}{(2-R)(1-R^{2})}=1/3=\varphi_{3}(-1)\in \partial \varphi_{3}(\mathbb{D}).\qedhere \]
\end{enumerate}
\end{proof}

\section{Radius Problem for $\mathcal{F}_{3}$}
If  the function   $f\in \mathcal{F}_{3}$,  then the function $p:\mathbb{D}\to\mathbb{C}$
defined by   $p(z)=f(z)/(z+z^2/2)$  is a function
in the class $\mathscr{P}(0)$   and
\begin{equation}\label{a2}
f(z)= p(z) (z+z^{2}/2)\quad (z\in \mathbb{D}).
\end{equation}
From \eqref{a2}, it follows that
\begin{equation}\label{b2}
{\frac{zf'(z)}{f(z)}=\frac{zp'(z)}{p(z)}+\frac{2(z+1)}{z+2}}.
\end{equation}
Using  \eqref{e} (with $\alpha=0$) and \eqref{c} in \eqref{b2},  we see that the image of the disk $\mathbb{D}_r$ under the mapping $zf'(z)/f(z)$  is contained in the disk
\begin{equation}\label{f2}
\left| \frac{zf'(z)}{f(z)}-\frac{4-2r^{2}}{4-r^{2}}\right|\leq \frac{2r(5-2r^{2})}{(1-r^{2})(4-r^{2})} \quad  (\left| z \right| \leq r).
\end{equation}
From \eqref{f2}, it readily follows that
\begin{equation}\label{g2}
\operatorname{Re} \left( \frac{zf'(z)}{f(z)}\right)\geq \frac{4-2r^{2}}{4-r^{2}}-\frac{2r(5-2r^{2})}{(1-r^{2})(4-r^{2})} =  \frac{2(1-3r+r^{3})}{(2-r)(1-r^{2})}  \quad  (\left| z \right| \leq r).
\end{equation}
Let  $R_{\mathcal{S}^{*}}\approx  0.3473$ be the unique zero in (0,1)
of the polynomial $r^{3}-3r+1$. Then,
for every function $f\in \mathcal{F}_{3}$,  the inequality \eqref{g1} shows that
$\operatorname{Re}(s_{f}(z))>0$ in
each disk $\mathbb{D}_{r}$, for $0\leq r< R_{\mathcal{S}^{*}}$.
For the function $f_{3}$ defined in \eqref{extremal1}, we have
\begin{equation}\label{th3b1}
s_{f_{3}}(z)= \frac{zf_{3}'(z)}{f_{3}(z)} =\frac{2(1+3z-z^3)}{(2+z)(1-z^{2})}
\end{equation}
and hence  $\operatorname{Re}(s_{f_{3}}(z))$ vanishes at $z=-R_{\mathcal{S}^{*}}$.
Thus, the radius of starlikeness  $R_{\mathcal{S}^{*}}$ for the class
$\mathcal{F}_{3}$ is the unique zero in (0,1)  of the polynomial
$P_{3}$ defined in \eqref{star3} and is the same as the radius
of univalence  $R_{\mathcal{S}}$. Using the inequality \eqref{f2}, we now determine     ${\mathcal{S^{*}}(\alpha)}$,
${\mathcal{{S}^{*}_{L}}}$, ${\mathcal{S}^{*}_{P}}$,
${\mathcal{S}^{*}_{e}}$, ${\mathcal{S}_{c}^{*}}$, ${\mathcal{S}_{sin}^{*}}$,
${\mathcal{S}_{\leftmoon}^{*}}$ and ${\mathcal{S}_{R}^{*}} $
radii  for the class $\mathcal{F}_{3}$.

\begin{theorem}\label{th3}
 The following sharp radii results hold for the class of function  $\mathcal{F}_{3}$:
\begin{enumerate}[label=(\roman*)]
\item For any  $0\leq \alpha<1$, the radius $R_{\mathcal{S^{*}}(\alpha)}$ is the smallest positive root  of the  polynomial \begin{equation}\label{1b} (2-\alpha)r^{3}+(2\alpha)r^{2}+(\alpha-6)r+2-2\alpha=0 .
\end{equation}
\item The radius $R_{\mathcal{{S}^{*}_{L}}}$ ($\approx 0.1645$) is the smallest positive root  of the  polynomial
\begin{equation}\label{2b}
(\sqrt{2}-2)r^{3}+(2\sqrt{2})r^{2}+(6-\sqrt{2})r+2-2\sqrt{2}=0.
\end{equation}
\item The radius $R_{\mathcal{S}^{*}_{p}}$ ($\approx 0.19028)$ is the same as  $R_{\mathcal{S}^{*}(1/2)}$.

\item The radius $R_{\mathcal{S}^{*}_{e}}$ ($ \approx 0.2355)$ is the same as  $R_{\mathcal{S}^{*}(1/e)}$.

\item The radius $R_{\mathcal{S}_{\sin}^{*}}$ ($ \approx 0.3017)$ is the same as  $R_{\mathcal{S}^{*}(1-\sin1)}$.

\item The radius $R_{\mathcal{S}_{\leftmoon}^{*}}$ ($\approx 0.2199)$ is the same as $R_{\mathcal{S}^{*}(\sqrt{2}-1)}$.

\item The radius $R_{\mathcal{S}_{R}^{*}}$ ($\approx 0.0679$) is  the same as $R_{\mathcal{S}^{*}(2(\sqrt{2}-1))}$.

\item The radius $R_{\mathcal{S}_{c}^{*}}$ ($ \approx 0.2469)$ is the same as  $R_{\mathcal{S}^{*}(1/3)}$.

\end{enumerate}

\end{theorem}
\begin{proof}
\begin{enumerate}[label=(\roman*),  leftmargin=12pt]
\item Let the function $f \in \mathcal{F}_{3}$ and $\alpha$ in [0,1). The root  $ R:=R_{\mathcal{S^{*}}(\alpha)}$ be the smallest positive root of the equation \eqref{1b} so that
\begin{equation}\label{th3p1}
2(1-3R+R^{3})=\alpha(2-R)(1-R^{2}).
\end{equation}
The function
\[h(r):=\frac{2(1-3r+r^{3})}{(2-r)(1-r^{2})}\]
is decreasing in $[0,1)$ and hence, for $ 0\leq r< R$,
we have, using \eqref{g2} and \eqref{th3p1},
\begin{equation*}
\operatorname{Re}\left(\frac{zf'(z)}{f(z)}\right)\geq\frac{2(1-3r+r^{3})}{(2-r)(1-r^{2})}>\frac{2(1-3R+R^{3})}{(2-R)(1-R^{2})}=\alpha.
\end{equation*}
This proves that the function $f$ is starlike of order $\alpha $
in each disk $\mathbb{D}_{r}$ for $0\leq r< R $.
At the point $z=-R$, it can be seen, using \eqref{th3b1} and  \eqref{th3p1},  that
the function $f_{3}$ defined in $\eqref{extremal1}$ satisfies
\begin{equation*}
\operatorname{Re}\left( zf_{3}'(z)/f_{3}(z)\right)=\frac{2(1-3R-R^3)}{(2-R)(1-R^{2})} =\alpha.
\end{equation*}
This shows that the radius $R$  is the  radius of
starlikeness of order $\alpha $ of the class $\mathcal{F}_{3}$.

\item Let $R:=R_{\mathcal{{S}^{*}_{L}}}$  be the smallest  positive root  of the equation \eqref{2b} so that
\begin{equation}\label{th3p2}
2(1+3R-R^{3})=\sqrt{2}(2+R)(1-R^{2}).
\end{equation}
Since the function
\[ h(r):=\frac{2r(5-2r^{2})}{(1-r^{2})(4-r^{2})}+\frac{4-2r^{2}}{4-r^{2}}=\frac{2(1+3r-r^{3})}{(1-r^{2})(2+r)}\]
is an increasing  function of $r$ in [0,1), it follows that
$h(r)<h(R)=\sqrt{2}$ for  $0\leq r < R$ and hence, for  $0\leq r < R$, we have
\begin{equation}\label{za}
 \frac{2r(5-2r^{2})}{(1-r^{2})(4-r^{2})}<\sqrt{2}-\frac{4-2r^{2}}{4-r^{2}}.
\end{equation}
From \eqref{f2} and \eqref{za}, we get
\begin{equation*}
\left|\frac{zf'(z)}{f(z)}-\frac{4-2r^{2}}{4-r^{2}} \right|<\sqrt{2}-\frac{4-2r^{2}}{4-r^{2}}  \quad  (\left| z \right| \leq r).
\end{equation*}
For  $0\leq r< R$, the center of the above disk
$c(r)= (4-2r^{2})/(4-r^{2})$ (being a decreasing function of $r$ on [0,1])
lies in the interval $[c(R),1]\subset (c(0.1), 1]
\approx (.997494,1]\subset [2\sqrt{2}/3,\sqrt{2})$. When
$a\in [2\sqrt{2}/3,\sqrt{2})$, by \cite[Lemma 2.2]{lemniscate}, the disk $|w-a|<\sqrt{2}-a$ is contained in the lemniscate region $|w^2-1|<1$ and
hence,  for  $0\leq r< R$,  we have
\begin{equation}\label{s1}
\left|\left({\frac{zf'(z)}{f(z)}}\right)^{2}-1\right|< 1.
\end{equation}
Thus, the radius of lemniscate starlikeness of the class $\mathcal{F}_3$
is at least $R$.  To show that the radius $R$ is sharp, using \eqref{th3b1} and \eqref{th3p2}, we see  that the function $f_{3}$ defined in $\eqref{extremal1}$ satisfies
\[ \frac{zf_{3}'(z)}{f_{3}(z)} =\frac{2(1+3R-R^{3})}{(1-R^{2})(2+R)}=\sqrt{2} \] at $z=-R$  and therefore
\begin{equation}
   \left|\left({\frac{zf_{3}'(z)}{f_{3}(z)}}\right)^{2}-1\right|= 1.
\end{equation}

\item The number $R:=R_{\mathcal{{S}^{*}_{P}}}$ is the smallest  positive root  of the equation
\begin{equation}\label{th3p3}
  2(1-3R+R^{3})=(1/2)(2-R)(1-R^{2}).
\end{equation}
Since the function
\[ h(r):=\frac{2r(2r^{2}-5)}{(1-r^{2})(4-r^{2})}+\frac{4-2r^{2}}{4-r^{2}}=\frac{2(1-3r+r^{3})}{(1-r^{2})(2-r)}\]
is decreasing  function of $r$ in [0,1), it follows that,
$h(r)>h(R)=1/2$ for  $0\leq r < R$ and hence, for  $0\leq r < R$, we have
\begin{equation}\label{zh}
\frac{2r(5-2r^{2})}{(1-r^{2})(4-r^{2})}<\frac{4-2r^{2}}{4-r^{2}}-\frac{1}{2}.
\end{equation}
From \eqref{f2} and \eqref{zh},  we get
\begin{align*}
\left|\frac{zf'(z)}{f(z)}-\frac{4-2r^{2}}{4-r^{2}} \right|<\frac{4-2r^{2}}{4-r^{2}}-\frac{1}{2}  \quad  (\left| z \right| \leq r).
\end{align*}
For  $0\leq r< R$, the center of the above disk
$c(r)= (4-2r^{2})/(4-r^{2})$ (being a decreasing function of $r$ on [0,1])
lies in the interval $[c(R),1]\subset (c(0.2), 1]
\approx (.9898991,1]\subset (1/2,3/2)$. When
$a\in (1/2,3/2)$, by \cite[Lemma 2.2]{shan},
the disk $|w-a|<a-(1/2)$ is contained in the parabolic region $|w-1|<\operatorname{Re}(w)$ and
hence,  for  $0\leq r< R$,  we have
\begin{align}\label{9b}
\left|{\frac{zf'(z)}{f(z)}}-1\right|<\operatorname{Re}\left(\frac{zf'(z)}{f(z)}\right) \quad  (\left| z \right| \leq r).
\end{align}
Thus, the radius of parabolic starlikeness of the class $\mathcal{F}_3$
is at least $R$.  To show that the radius $R$  is sharp, using \eqref{th3b1} and \eqref{th3p3}, we see  that the function $f_{3}$ defined in \eqref{extremal1} satisfies
\[ \frac{zf_{3}'(z)}{f_{3}(z)} =\frac{2(R^{3}-3R+1)}{(1-R^{2})(2-R)}=\frac{1}{2} \] at $z=-R$  and therefore
\begin{align}
\left|{\frac{zf_{3}'(z)}{f_{3}(z)}}-1\right|=\frac{1}{2}=\operatorname{Re}\left({\frac{zf_{3}'(z)}{f_{3}(z)}} \right).
\end{align}

\item The number $R:=R_{\mathcal{S}^{*}_{e}}$ is the smallest  positive root  of the equation
\begin{equation}\label{th3p4}
 2(1-3R+R^{3})=(1/e)(2-R)(1-R^{2}).
\end{equation}
Since the function
\[ h(r):=\frac{2r(2r^{2}-5)}{(1-r^{2})(4-r^{2})}+\frac{4-2r^{2}}{4-r^{2}}=\frac{2(1-3r+r^{3})}{(1-r^{2})(2-r)}\]
is decreasing  function of $r$ in [0,1), it follows that
$h(r)>h(R)=1/e$ for  $0\leq r < R$ and hence, for  $0\leq r < R$, we have
\begin{equation}\label{zc}
\frac{2r(5-2r^{2})}{(1-r^{2})(4-r^{2})}<\frac{4-2r^{2}}{4-r^{2}}-\frac{1}{e}.
\end{equation}
From \eqref{f2} and \eqref{zc}, we get
\begin{align*}
\left|\frac{zf'(z)}{f(z)}-\frac{4-2r^{2}}{4-r^{2}} \right|<\frac{4-2r^{2}}{4-r^{2}}-\frac{1}{e}  \quad  (\left| z \right| \leq r).
\end{align*}
For  $0\leq r< R$, the center of the above disk
$c(r)= (4-2r^{2})/(4-r^{2})$ (being a decreasing function of $r$ on [0,1]) lies in the interval $[c(R),1]\subset (c(0.3), 1] \approx (.976982,1]\subset (e^{-1},(e+e^{-1})/2$]. When
$a\in (e^{-1},(e+e^{-1})/2]$, by \cite[Lemma 2.2]{rajni},
the disk $|w-a|<a-e^{-1}$ is contained in the  region $|\log w|<1$ and hence, for  $0\leq r< R$,  we have
\begin{align}\label{10b}
\left|\operatorname{log}\left(\frac{zf'(z)}{f(z)}\right) \right|<1  \quad  (\left| z \right| \leq r).
\end{align}
Thus, the radius of exponential starlikeness of the class $\mathcal{F}_3$
is at least $R$.  To show that the radius $R$  is sharp, using \eqref{th3b1} and \eqref{th3p4}, we see  that the function
$f_{3}$ defined in \eqref{extremal1} satisfies
\[ \frac{zf_{3}'(z)}{f_{3}(z)} =\frac{2(1-3R+R^{3})}{(1-R^{2})(2-R)}=\frac{1}{e} \] at $z=-R$  and therefore
\begin{align}
\left|\operatorname{log}\left(\frac{zf_{3}'(z)}{f_{3}(z)}\right) \right|=1.
\end{align}

\item The number $R:=R_{\mathcal{S}^{*}_{\sin}}$  is the smallest  positive root  of the equation
\begin{equation}\label{th3p5}
2(1-3R+R^{3})=(1-\sin1)(2-R)(1-R^{2}).
\end{equation}
Since the function
\[ h(r):=\frac{2r(2r^{2}-5)}{(1-r^{2})(4-r^{2})}+\frac{4-2r^{2}}{4-r^{2}}=\frac{2(1-3r+r^{3})}{(1-r^{2})(2-r)}\]
is decreasing  function of $r$ in [0,1), it follows that
$h(r)>h(R)=1-\sin1$ for  $0\leq r < R$ and hence, for  $0\leq r < R$, we have
\begin{equation}\label{zd}
\frac{2r(5-2r^{2})}{(1-r^{2})(4-r^{2})}<\frac{4-2r^{2}}{4-r^{2}}+\sin1-1.
\end{equation}
From \eqref{f2} and \eqref{zd}, get
\begin{align*}
\left|\frac{zf'(z)}{f(z)}-\frac{4-2r^{2}}{4-r^{2}} \right|<\sin1-1+\frac{4-2r^{2}}{4-r^{2}}  \quad  (\left| z \right| \leq r).
\end{align*}
For  $0\leq r< R$, the center of the above disk
$c(r)= (4-2r^{2})/(4-r^{2})$ (being a decreasing function of $r$ on [0,1])
lies in the interval $[c(R),1]\subset (c(0.4), 1]
\approx (.995833,1]\subset (-1-\sin1, 1-\sin1)$. When
$a\in (-1-\sin1, 1+\sin1)$, by \cite[Lemma 3.3]{cho},
the disk $|w-a|<\sin1-|a-1|$ is contained in the region $\varphi_{4}(\mathbb{D})$, where $\varphi_{4}(z)=1+\sin z$ and
hence,  for  $0\leq r< R$, ${s_{f}}(\mathbb{D}_{r})\subset\varphi_{4}(\mathbb{D})$.
Thus, the radius of sine starlikeness of the class $\mathcal{F}_3$
is at least $R$.  To show that the radius $R$  is sharp, using \eqref{th3b1} and \eqref{th3p5}, we see  that the function $f_{3}$ defined in \eqref{extremal1} satisfies, at $z=-R$,
\[ \frac{zf_{3}'(z)}{f_{3}(z)} =\frac{2(1-3R+R^{3})}{(2+R)(1-R^{2})}=1-\sin1=\varphi_{4}(-1)\in \partial \varphi_{4}(\mathbb{D}). \]

\item The number $R:=R_{\mathcal{S}^{*}_{\leftmoon}}$ is the smallest  positive root  of the equation
\begin{equation}\label{th3p6}
2(1-3R+R^{3})=(\sqrt{2}-1)(2-R)(1-R^{2}).
\end{equation}
Since the function
\[ h(r):=\frac{2r(2r^{2}-5)}{(1-r^{2})(4-r^{2})}+\frac{4-2r^{2}}{4-r^{2}}=\frac{2(1-3r+r^{3})}{(1-r^{2})(2-r)}\]
 is decreasing  function of $r$ in [0,1), it follows that
 $h(r)>h(R)=\sqrt{2}-1$ for  $0\leq r < R$ and hence, for  $0\leq r < R$, we have
\begin{equation}\label{ze}
\frac{2r(5-2r^{2})}{(1-r^{2})(4-r^{2})}<\frac{4-2r^{2}}{4-r^{2}}+1-\sqrt{2}.
\end{equation}
 From \eqref{f2} and \eqref{ze}, we get
\begin{align*}
\left|\frac{zf'(z)}{f(z)}-\frac{4-2r^{2}}{4-r^{2}} \right|<\frac{4-2r^{2}}{4-r^{2}}+1-\sqrt{2}  \quad  (\left| z \right| \leq r).
\end{align*}
For  $0\leq r< R$, the center of the above disk
 $c(r)= (4-2r^{2})/(4-r^{2})$ (being a decreasing function of $r$ on [0,1])
 lies in the interval $[c(R),1]\subset (c(0.3), 1]
\approx (.976982,1]\subset (\sqrt{2}$-1,$\sqrt{2}+1$). When
$a\in(\sqrt{2}-1,\sqrt{2}+1)$, by \cite[Lemma 2.1]{gandhi},
the disk $|w-a|<1-|\sqrt{2}-a|$ is contained in the  region $|w^{2}-1|<2|w|$ and hence,  for  $0\leq r< R$,  we have
\begin{align}\label{11b}
\left|\left({\frac{zf'(z)}{f(z)}}\right)^{2}-1\right|< 2\left|\left({\frac{zf'(z)}{f(z)}}\right)\right|  \quad  (\left| z \right| \leq r).
\end{align}
Thus, the radius of lune starlikeness of the class $\mathcal{F}_3$ is at least $R$.  To show that the radius $R$  is sharp, using \eqref{th3b1} and \eqref{th3p6}, we see  that the function $f_{3}$ defined in \eqref{extremal1} satisfies
\[ \frac{zf_{3}'(z)}{f_{3}(z)} =\frac{2(1-3R+R^{3})}{(1-R^{2})(2-R)}=\sqrt{2}-1 \]
at $z=-R$  and therefore
\begin{align}
\left|\left({\frac{zf_{3}'(z)}{f_{3}(z)}}\right)^{2}-1\right|= 2\left|\left({\frac{zf_{3}'(z)}{f_{3}(z)}}\right)\right|.
\end{align}

\item The number $R:=R_{\mathcal{S}^{*}_{R}}$  is the smallest  positive root  of the equation
\begin{equation}\label{th3p7}
2(1-3R+R^{3})=2(\sqrt{2}-1)(2-R)(1-R^{2}).
\end{equation}
Since the function
  \[ h(r):=\frac{2r(2r^{2}-5)}{(1-r^{2})(4-r^{2})}+\frac{4-2r^{2}}{4-r^{2}}=\frac{2(1-3r+r^{3})}{(1-r^{2})(2-r)}\]
is decreasing  function of $r$ in [0,1), it follows that
$h(r)>h(R)=2(\sqrt{2}-1)$ for  $0\leq r < R$ and hence, for  $0\leq r < R$, we have
\begin{equation}\label{zf}
\frac{2r(5-2r^{2})}{(1-r^{2})(4-r^{2})}<\frac{4-2r^{2}}{4-r^{2}}-2(\sqrt{2}-1).
\end{equation}
From \eqref{f2} and \eqref{zf}, we get
\begin{align*}
\left|\frac{zf'(z)}{f(z)}-\frac{4-2r^{2}}{4-r^{2}} \right|<\frac{4-2r^{2}}{4-r^{2}}-2(\sqrt{2}-1)  \quad  (\left| z \right| \leq r).
\end{align*}
For  $0\leq r< R$, the center of the above disk $c(r)= (4-2r^{2})/(4-r^{2})$ (being a decreasing function of $r$ on [0,1])
lies in the interval $[c(R),1]\subset (c(0.1), 1] \approx (.99741,1]\subset (2(\sqrt{2}-1), \sqrt{2}]$. When
$a\in  (2(\sqrt{2}-1), \sqrt{2}]$, by \cite[Lemma 2.2]{sushil},
the disk $|w-a|<a-2(\sqrt{2}-1)$ is contained in the region $\varphi_{6}(\mathbb{D})$, where $\varphi_{6}(z):=1+((zk+z^2)/(k^2-kz))$ and $k=\sqrt{2}+1$.
Hence,  for  $0\leq r< R$, ${s_{f}}(\mathbb{D}_{r})\subset\varphi_{6}(\mathbb{D})$.
Thus, the $\mathcal{S}^{*}_{R}$ radius of the class $\mathcal{F}_3$
is at least $R$.  To show that the radius $R$  is sharp, using \eqref{th3b1} and \eqref{th3p6}, we see  that the function
$f_{3}$ defined in \eqref{extremal1} satisfies, at $z=-R$,
\[ \frac{zf_{3}'(z)}{f_{3}(z)} =\frac{2(1-3R+R^{3})}{(2-R)(1-R^{2})}=2(\sqrt{2}-1)=\varphi_{6}(-1)\in \partial \varphi_{6}(\mathbb{D}). \]

\item The number $R:=R_{\mathcal{S}^{*}_{C}}$ is the smallest  positive root  of the equation
\begin{equation}\label{th3p8}
2(1-3R+R^{3})=(1/3)(2-R)(1-R^{2}).
\end{equation}
Since the function
\[ h(r):=\frac{2r(2r^{2}-5)}{(1-r^{2})(4-r^{2})}+\frac{4-2r^{2}}{4-r^{2}}=\frac{2(1-3r+r^{3})}{(1-r^{2})(2+r)}\] is decreasing  function of $r$ in [0,1), it follows that $h(r)>h(R)=1/3$ for  $0\leq r < R$ and hence, for  $0\leq r < R$, we have
\begin{equation}\label{zb}
\frac{2r(5-2r^{2})}{(1-r^{2})(4-r^{2})}<\frac{4-2r^{2}}{4-r^{2}}-\frac{1}{3}.
\end{equation}
From \eqref{f2} and\eqref{zb}, we get
\begin{align*}
\left|\frac{zf'(z)}{f(z)}-\frac{4-2r^{2}}{4-r^{2}} \right|<\frac{4-2r^{2}}{4-r^{2}}-\frac{1}{3}  \quad  (\left| z \right| \leq r).
\end{align*}
For  $0\leq r< R$, the center of the above disk
$c(r)= (4-2r^{2})/(4-r^{2})$ (being a decreasing function of $r$ on [0,1])
lies in the interval $[c(R),1]\subset (c(0.3), 1]
\approx (.976982,1]\subset (1/3,5/3)$. When
$a\in  (1/3,5/3)$, by \cite[Lemma 2.5]{kanika},
the disk $|w-a|<a-1/3$ is lies in the cardioid region $\varphi_{3}(\mathbb{D})$.
Hence,  for  $0\leq r< R$,
${s_{f}}(\mathbb{D}_{r})\subset\varphi_{3}(\mathbb{D})$.
Thus, the radius of the class $\mathcal{F}_3$
is at least $R$.  To show that the radius $R$  is sharp, using \eqref{th3b1} and \eqref{th3p8}, we see  that the function
$f_{3}$ defined in \eqref{extremal} satisfies, at $z=-R$,
\[ \frac{zf_{3}'(z)}{f_{3}(z)} =\frac{2(1-3R+R^{3})}{(2-R)(1-R^{2})}=1/3=\varphi_{3}(-1)\in \partial \varphi_{3}(\mathbb{D}). \qedhere\]
\end{enumerate}
\end{proof}

\end{document}